\documentclass[letterpaper, 10 pt, conference]{ieeeconf}

\IEEEoverridecommandlockouts
\usepackage{cite}

\usepackage{amsthm,amsmath,amssymb,amsfonts}
\usepackage{bm}
\usepackage{algorithm}
\usepackage{algpseudocode}

\usepackage{graphicx}
\usepackage{xcolor}
\usepackage{hyperref}

\usepackage{array}
\usepackage{multirow}

\usepackage[normalem]{ulem}

\makeatletter
\let\MYcaption\@makecaption
\makeatother

\usepackage[font=footnotesize]{subcaption}

\makeatletter
\let\@makecaption\MYcaption
\makeatother

\usepackage{textcomp}
\def\BibTeX{{\rm B\kern-.05em{\sc i\kern-.025em b}\kern-.08em
T\kern-.1667em\lower.7ex\hbox{E}\kern-.125emX}}

\usepackage{lipsum}


\theoremstyle{plain}
\newtheorem{theorem}{Theorem}
\newtheorem{lemma}{Lemma}

\theoremstyle{definition}

\newtheorem{definition}{Definition}
\newtheorem{assumption}{Assumption}
\newtheorem{remark}{Remark}

\def\({\left(}
\def\){\right)}
\def\[{\left[}
\def\]{\right]}

\newif\ifshowWriterComment
\newcommand\writercomment[3]{\expandafter}

\DeclareMathOperator*{\argmin}{argmin}

\def\alg#1{Algorithm~\ref{alg:#1}}
\def\fig#1{Fig.~\ref{fig:#1}}
\def\lem#1{Lemma~\ref{lem:#1}}
\def\sec#1{Section~\ref{sec:#1}}
\def\tab#1{Table~\ref{tab:#1}}
\def\thm#1{Theorem~\ref{thm:#1}}
\def\eqn#1{\eqref{eqn:#1}}

\def\st{{\rm s.t.}}

\newcommand{\OptMinStar}[3][]{
\ifx\\#1\\ \else\begin{subequations}\label{eqn:#1}\fi%
\begin{alignat}{2}
\min\ &\ #2 \nonumber \\
\st\ #3
\end{alignat}
\ifx\\#1\\ \else\end{subequations}\fi%
}
\newcommand{\OptMinNoStar}[3][]{
\ifx\\#1\\ \else\begin{subequations}\label{eqn:#1}\fi%
\begin{alignat}{2}
\min\ &\ #2 \ifx\\#1\\ \nonumber \else \tag{\ref{eqn:#1}}\label{eqnset:#1} \fi \\
\st\ #3
\end{alignat}
\ifx\\#1\\ \else\end{subequations}\fi%
}

\title{\LARGE \bf Robust Distributed and Localized Model Predictive Control}

\author{Carmen Amo Alonso, Jing Shuang (Lisa) Li, Nikolai Matni and James Anderson
\thanks{C. Amo Alonso and J.S. Li are with the Computing and Mathematical Sciences Department at California Institute of Technology, Pasadena, CA, {\tt\small \{camoalon,jsli\}@caltech.edu}.  N. Matni is with the Department of Electrical and Systems Engineering at the University of Pennsylvania, Philadelphia, PA, {\tt\small nmatni@seas.upenn.edu}. J. Anderson is with the Department of Electrical Engineering and the Data Science Institute at Columbia University, New York, NY. {\tt\small james.andersosn@columLbia.edu}.  N. Matni is generously supported by NSF award CPS-2038873 and CAREER award ECCS-2045834.}}

\showWriterCommenttrue

\begin{document}

\maketitle
\thispagestyle{empty}
\pagestyle{empty}

\bstctlcite{IEEE_BSTcontrol}

\begin{abstract}

We present a robust Distributed and Localized Model Predictive Control (rDLMPC) framework for large-scale structured linear systems. The proposed algorithm uses the System Level Synthesis  to provide a distributed \emph{closed-loop} model predictive control scheme that is robust to exogenous disturbances. The resulting controllers require  only local information  exchange for both synthesis and implementation. We exploit the fact that for polytopic disturbance constraints, SLS-based distributed control problems have been shown to have structure amenable for distributed optimization techniques. We show that similar to the disturbance-free DLMPC algorithm, the computational complexity of rDLMPC  is independent of the size of the global system. To the best of our knowledge, robust DLMPC is the first MPC algorithm that allows for the scalable distributed computation of distributed closed-loop control policies in the presence of additive disturbances.

\end{abstract}

\section{Introduction}\label{sec:introduction}

Model Predictive Control (MPC) has proven to be very successful in many industrial applications. It provides an effective way to control systems that are subject to hard state and input constraints while ensuring good performance. However, one key consideration when dealing with constrained control in real-world applications is the presence of disturbances. While nominal MPC enjoys some intrinsic robustness \cite{limonmarruedo2002stability}, the resulting closed-loop can be destabilized by an arbitrary small disturbance \cite{grimm2004examples}. For this reason, efforts have been made to extend nominal MPC approaches to the robust setting.

Our focus is on designing robust distributed MPC controllers for heterogenous networks with additive disturbances. Broadly speaking, there are two existing methods to tackling this problem:  the first builds on centralized robust MPC techniques based on a pre-computed stabilizing controller (tube-MPC \cite{langson2004robust}, constraint tightening \cite{richards2006robust}, etc.), and ports them to the distributed setting. Although this extension results in computationally efficient algorithms, they are usually too restrictive in their assumptions: for example, the subsystems must be dynamically decoupled  \cite{richards2007robust}, or, a pre-computed structured stabilizing controller is needed \cite{conte2013robust} (which requires solving a NP-hard problem \cite{blondel1997np}). The second approach is based on computing a dynamic structured feedback policy using a suitable parameterization. In the seminal paper \cite{goulart2006optimization}, the problem of synthesizing robust control policies is solved through convex optimization by virtue of using a disturbance based parameterization of the control policy. Other works have been developed based on these ideas \cite{furieri2017robust}, and although they are able to formulate the synthesis of a structured closed-loop control policy as a convex problem, the resulting optimization is not amenable for distributed optimization techniques.
Along this line, recent works exploit the System Level Synthesis (SLS) parametrization to design robust MPC controllers \cite{sieber2021system,chen2020robust}. However, it is not yet clear how these formulations can be applied in a distributed setting.

Despite these efforts,  synthesizing \emph{and} implementing a closed-loop MPC controller that is robust to additive disturbances in a distributed fashion remains an open problem.  For large-scale networks, it is desirable for the control scheme to respect communication limitations among sub-controllers. Here we present an algorithm based on the SLS framework that achieves both these goals.  To do so, we exploit the SLS parameterization, which characterizes achievable closed-loop system responses as linear maps from disturbances to input and states, and integrate our prior results on Distributed and Localized MPC (DLMPC) \cite{amoalonso2020distributed} and robust invariant linear-time-invariant control \cite{chen2019system} into a novel robust Distributed and Localized MPC (rDLMPC) algorithm.

Our rDLMPC algorithm provides a convex formulation, has robust stability guarantees, and has local (i.e., per subsystem) computational and implementation complexity that scales independently of the number of states in the system model. Moreover, it allows for the incorporation of locality constraints, where information exchanges are restricted to local sub-controllers only. The concept of locality is central to the SLS framework: it permits scalable synthesis and algorithms and is a desirable property for the closed-loop system. Broadly speaking, a localized system response limits the impact of a disturbance affecting a subsystem in the network to a small neighborhood of connected subsystems. We formalize this idea in \sec{sls}. A by-product of localization is that the resulting MPC sub-controllers only have to solve a low-dimensional optimization problem, and only exchange information with controllers in a local neighborhood. We demonstrate through simulation that the computational overhead and decrease in performance due to conservatism is only slightly higher than for the nominal case \cite{amoalonso2020distributed}, and that  the complexity of the subproblems solved by each subsystem is independent of the size of the system.

\paragraph*{Notation}

Bracketed indices denote time-step of the real system, i.e., the system input is  $u(t)$ at time $t$, not to be confused with $x_t$ which denotes the predicted state $x$ at time $t$. Superscripted variables, e.g. $x^k$,  correspond to the value of $x$ at the $k^{th}$ iteration of a given algorithm. Square bracket notation, i.e., $[x]_{i}$ denotes the components of $x$ corresponding to subsystem $i$.  Calligraphic letters such as $\mathcal{S}$ denote sets, and lowercase script letters such as $\mathfrak{c}$ denotes a subset of $\mathbb{Z}^{+}$, i.e., $\mathfrak{c}=\left\{1,...,n\right\}\subset\mathbb{Z}^{+}$.  Boldface lower and upper case letters such as $\mathbf{x}$ and $\mathbf{K}$ denote finite horizon signals and block lower triangular (causal) operators respectively. We define
\begin{equation*}
\mathbf{x}=\left[\begin{array}{c} x_{0}\\x_{1}\\\vdots\\x_{T}\end{array}\right], ~
K =   { {\scriptscriptstyle{\left[\begin{array}{cccc}K_{0}[0] & & & \\ K_{1}[1] & K_{1}[0] & & \\ \vdots & \ddots & \ddots & \\ K_{T}[T] & \dots & K_{T}[1] & K_{T}[0] \end{array}\right]}}},
\end{equation*}
where each $x_i$ is an $n$-dimensional vector, and each $K_{i}[j]$ is a matrix of compatible dimension representing the value of $K$ at the $j^\text{th}$ time-step computed at time $i$. $\mathbf{K}(\mathfrak{r},\mathfrak{c})$ denotes the submatrix of $\mathbf{K}$ composed of the rows and columns specified by $\mathfrak{r}$ and $\mathfrak{c}$ respectively. We denote the block columns of  $\mathbf K$ by $\mathbf K\{1\}$,...,$\mathbf K\{T\}$, i.e.  $\mathbf K\left\{1\right\}:=[K_{0}[0]^{\mathsf{T}}\ \dots\ K_{T}[T]^{\mathsf{T}}]^{\mathsf{T}}$, and we use $:$ to indicate the range of columns, i.e. $\mathbf K\left\{2:T\right\}$ contains the block columns from the second to the last.

\section{Problem Statement}\label{sec:problem_statement}

We consider a discrete-time linear time invariant (LTI) system, with dynamics:
\begin{equation}\label{eqn:LTI}
x(t+1) = Ax(t)+Bu(t)+w(t),
\end{equation}
where $x(t)\in \mathbb{R}^n$ and $u(t)\in\mathbb{R}^p$ are the state and control input respectively, and $w(t)\in\mathbb{R}^n$ is an exogenous disturbance. We view system \eqn{LTI} as being composed of $N$ interconnected subsystems, so the state, control input and disturbance can be partitioned into $[x(t)]_i$, $[u(t)]_i$ and $[w(t)]_i$ for each subsystem $i$ inducing a compatible block structure $[A]_{ij}$, $[B]_{ij}$ in the dynamic matrices $(A,B)$. We model the interconnection topology as a time-invariant unweighted directed graph $\mathcal{G}_{(A,B)}(\mathcal{E},\mathcal{V})$, where each subsystem $i$ is identified with a vertex $v_i\in \mathcal{V}$ and an edge between subsystems $i$ and $j$ $e_{ij}\in\mathcal{E}$ exists whenever $[A]_{ij}\neq0$ or $[B]_{ij}\neq0$. We further assume that the information exchange topology between sub-controllers matches that of the underlying system, and can thus be modeled by the same graph $\mathcal{G}_{(A,B)}(\mathcal{E},\mathcal{V})$.

We study the case where the control input is a model predictive controller optimizing a nominal objective, subject to constraints that must be robustly satisfied for all realizations of exogenous disturbances within a set $w(t)\in\mathcal{W}$. As is standard, at each time step the controller solves an optimal control problem over a finite prediction horizon $T$ using the current state as the initial condition. Hence, at time step $\tau$, we solve:
\begin{equation} \label{eqn:MPC}
\begin{aligned}
& \underset{{x}_{t},u_{t}, \gamma_t}{\text{minimize}} &  &\sum_{t=0}^{T-1}f_{t}(x_{t},u_{t})+f_{T}(x_{T})\\
& \ \text{s.t.} &  &\begin{aligned}
& x_{0} = x(\tau),\, \ x_{t+1} = Ax_{t}+Bu_{t}+w_t,\\
& x_{T}\in\mathcal{X}_{T},\, x_{t}\in\mathcal{X}_{t},\, u_{t}\in\mathcal{U}_{t},\, \forall w_t\in\mathcal{W}_t \\
&u_{t} = \gamma_t(x_{0:t},u_{0:t-1}), \, t=0,...,T-1,
\end{aligned}
\end{aligned}
\end{equation}
where the $f_t(\cdot,\cdot)$ and $f_T(\cdot)$ are convex cost functions, $\mathcal{X}_t$, $\mathcal{U}_t$ and $\mathcal{W}_t$ are convex polytopes containing the origin, and $\gamma_t(\cdot)$ is the control policy we optimize over.

Our goal is to define an algorithm that allows us to solve the robust MPC problem \eqn{MPC} in a distributed manner while respecting local communication constraints.  To achieve this goal, we impose that information exchange - as defined by the graph $\mathcal{G}_{(A,B)}(\mathcal{E},\mathcal{V})$ - is \emph{localized} to a subset of neighboring sub-controllers. Analogous to previous work \cite{amoalonso2020distributed,amoalonso2020explicit}, we use the notion of  a $d$-local information exchange constraint via the $d$-outgoing and $d$-incoming sets of the subsystems to restrict sub-controllers exchanges of their state and control actions to only neighbors at most $d$-hops away.
\begin{definition}{\label{def: in_out set}}
For a graph $\mathcal{G}(\mathcal{V},\mathcal{E})$, the \textit{d-outgoing set} of subsystem $i$ is $\textbf{out}_{i}(d) := \left\{v_{j}\ |\  \textbf{dist}(v_{i} \rightarrow v_{j} ) \leq d\in\mathbb{N} \right\}$. The \textit{d-incoming set} of subsystem $i$ is $\textbf{in}_{i}(d) := \left\{v_{j}\ |\ \textbf{dist}(v_{j} \rightarrow v_{i} ) \leq d\in\mathbb{N} \right\}$.  Note that $v_i \in \textbf{out}_{i}(d)\cap \textbf{in}_{i}(d)$ for all $d\geq 0$.
\end{definition}

Hence, we can achieve $d$-local information exchange on the distributed MPC problem \eqn{MPC}  by enforcing the following constraint on each local control policy:
\begin{multline}
[u_t]_i = \gamma_{t}^i\left([x_{0:t}]_{j\in \textbf{in}_i(d)},[u_{0:t-1}]_{j\in \textbf{in}_i(d)},\right.\\\left.
[A]_{j,k \in \textbf{in}_i(d)},[B]_{j,k\in  \textbf{in}_i(d)}\}\right),
\label{eqn:local_comms}
\end{multline}
for all $t=0,\dots,T$ and $i=1,\dots,N$, where $\gamma_{t}^i$ is a measurable function of its arguments. This implies that the closed-loop control policy for sub-controller $i$ can be  be computed using only states, control actions, and system models collected from $d$-hop incoming neighbors according to the communication topology $\mathcal{G}_{(A,B)}.$

In the remaining sections of this paper, we will show that by exploiting this information exchange topology together with suitable structural compatibility assumptions between the cost function and state and input constraints, we achieve an MPC controller robust to additive disturbances in which both synthesis and implementation of a control action at each subsystem are localized, i.e., depend only on state, control action, and plant model information from $d$-hop neighbors, where the size of the local neighborhood $d$ is a design parameter. Analogous to \cite{amoalonso2020distributed}, we achieve this by imposing appropriate $d$-local structural constraints on the closed-loop system responses of the system by leveraging the SLS framework to reformulate the MPC problem \eqn{MPC}. In addition, in this work we take advantage of recent robustness results in constrained SLS \cite{chen2019system}, which allows us to extend the results in \cite{amoalonso2020distributed} to the robust setting.

\section{System Level Synthesis MPC Reformulation}\label{sec:sls}

In this section we introduce relevant concepts of the SLS framework \cite{anderson2019system,wang2019system,wang2018separable} such as locality \cite{wang2018separable,wang2014localized} and how to encode state and input constraints \cite{chen2019system}, and discuss their central role in enabling a rDLMPC framework. We end by using these tools to formulate the robust Distributed and Localized MPC problem in the SLS framework.

\subsection{Time Domain System Level Synthesis}

Given the dynamics of system \eqn{LTI}, let $u(t) = K_t(x(0),...,x(t))$ be the control input from a causal linear\footnote{Without loss we assume the control policy is linear, as an affine control policy $u(t) = K_t(x(0:t))+v_t$ can always be written as a linear policy acting on the augmented state $\tilde{x} = [x^\intercal ~ 1^\intercal]^\intercal$.} time varying state feedback controller where $K_t$ is to be designed. The closed-loop dynamics over a finite time horizon $t = 0,...,T$ can be compactly written as
\begin{equation*}
\mathbf x = Z(\hat A+ \hat B\mathbf K)\mathbf x +\mathbf w,
\end{equation*}
where $Z$ is the block-downshift matrix\footnote{A matrix with identity matrices along its first block sub-diagonal and zeros elsewhere.}, $\hat A := \textbf{blkdiag}(A,A,...,A)$, and $\hat B := \textbf{blkdiag}(B,B,...,B)$.  The block lower-triangular operator $\mathbf{K}$ represents the control law defined by the linear-time-varying gains $\{K_t\}_{t=0}^T$, and the finite horizon signals $\mathbf{x},\ \mathbf{u},\ \mathbf{w}$ correspond to the state, control input and disturbance respectively.  By convention, we define the disturbance to contain the initial condition, i.e. $\mathbf{w}=[x_{0}^\intercal,w_{1}^\intercal,...,w_{T-1}^\intercal]^\intercal$.

By noting that $\mathbf{u}=\mathbf{K}\mathbf{x}$, the closed-loop behavior of the system can entirely be characterized by the block lower-triangular operators $\mathbf \Phi_x$ and $\mathbf \Phi_u$ (mapping disturbance to state and control action, respectively) defined as:
\begin{equation} \label{eqn:Phis}
\begin{split}
\mathbf{x} & = (I-Z(A+B\mathbf{K}))^{-1}\mathbf{w} =: \mathbf\Phi_x \mathbf w\\
\mathbf{u} & = \mathbf{K}(I-Z(A+B\mathbf{K}))^{-1}\mathbf{w} =: \mathbf\Phi_u \mathbf w.
\end{split}
\end{equation}
The \emph{system responses} $\mathbf{\Phi}_x$ and $\mathbf{\Phi}_u$ parameterize the set of achievable closed-loop maps \eqn{Phis} from the disturbance $\mathbf{w}$ to state $\mathbf{x}$  and control input $\mathbf{u}$ respectively, if and only if they are constrained to lie in the affine subspace
\begin{equation}\label{eqn:ZAB}
Z_{AB}\mathbf{\Phi}:=\begin{bmatrix}I-Z\hat{A} & -Z\hat{B}\end{bmatrix}\begin{bmatrix} \mathbf{\Phi}_x \\ \mathbf{\Phi}_u\end{bmatrix} = I,
\end{equation}
where we use $Z_{AB}$ and $\mathbf{\Phi}$ as short-hand notation for the constraint matrix and joint state/input system response, respectively.
In particular, any system responses $\mathbf{\Phi}_x$ and $\mathbf{\Phi}_u$ satisfying equation \eqref{eqn:ZAB} can be achieved via the control law $\mathbf{K}(\mathbf{\Phi}_x,\mathbf{\Phi}_u)=\mathbf{\Phi}_u\mathbf{\Phi}_x^{-1}$, where $\mathbf{K}(\mathbf{\Phi}_x,\mathbf{\Phi}_u)$ is also block lower triangular. For a formal statement and a proof, please see \cite[Theorem 2.1]{anderson2019system}.

The SLS framework relies on the affine parametrization \eqn{Phis} to reformulate optimal control problems as a search over system responses $\mathbf{\Phi}$ satisfying equation \eqn{ZAB}, rather than an optimization problem over states and inputs $\{\mathbf x, \mathbf u \}$. Using parametrization \eqn{Phis}, we reformulate the MPC subroutine \eqn{MPC} in terms of the system responses as
\begin{equation} \label{eqn:MPC_SLS}
\begin{aligned}
& \underset{\mathbf{\Phi}}{\text{minimize}}
& &f(\mathbf{\Phi}\{1\}x_0)\\
& \ \text{s.t.} &  &\begin{aligned}
& x_0 = x(\tau),\  Z_{AB}\mathbf{\Phi} = I,\ \mathbf{\Phi}\mathbf{w}\in\mathcal{P} \ \  \forall \mathbf{w}\in\mathcal{W},
\end{aligned}
\end{aligned}
\end{equation}
where the polytope $\mathcal{W} := \otimes_{t=0}^{T} \mathcal{W}_t$ and the polytope $\mathcal P$ is defined so that $\mathbf{\Phi}\mathbf{w}\in\mathcal{P}$ if and only if $x_{T}\in\mathcal{X}_{T},\, x_{t}\in\mathcal{X}_{t},\text{ and } u_{t}\in\mathcal{U}_{t}$, for all $t=0,...,T-1,\ \forall w\in\mathcal{W}$.  Note that $\mathbf \Phi\{1\} x_0$ appears in the objective function as it corresponds to the nominal (disturbance-free) state and input responses.

In what follows we discuss the advantages of this parametrization, and show how local structure can be imposed in this reformulation in a transparent manner via affine constraints. We also discuss how one can deal with the robust constraints while maintaining the local structure of the problem.

\subsubsection{Locality}

Here we illustrate how to enforce the information sharing constraint \eqn{local_comms} in the SLS framework, and how localized system responses result in a localized controller implementation.

A key advantage of using the SLS framework is that the system responses not only parametrize the closed-loop map but also provide a controller realization. In particular, the controller achieving the system responses \eqn{Phis} can be implemented as
\begin{equation}\label{eqn:implementation}
\begin{aligned}
\mathbf{u}=\mathbf{\Phi}_u\mathbf{\hat{w}},\ \ \  \mathbf{\hat{x}}=(I - \mathbf{\Phi}_x)\mathbf{\hat{w}},\ \ \
\mathbf{\hat{w}}=\mathbf{x}-\mathbf{\hat{x}},
\end{aligned}
\end{equation}
where $\mathbf{\hat{x}}$ is the nominal state trajectory, and $\mathbf{\hat{w}}=Z\mathbf{w}$ is a delayed reconstruction of the disturbance. The advantage of this controller implementation, as opposed to  $\mathbf u = \mathbf \Phi_u \mathbf \Phi_x^{-1} \mathbf x$, is that any structure imposed on the system response $\{\mathbf\Phi_u, \mathbf \Phi_x\}$ translates directly to structure on the controller implementation \eqn{implementation}. This is particularly relevant for imposing locality constraints, and we will show how locality in system responses translates into locality of the controller implementation.

\begin{definition}{\label{def: locality}}
Let $[\mathbf{\Phi}_{x}]_{ij}$ be the submatrix of system response $\mathbf{\Phi}_x$ describing the map from disturbance $[w]_{j}$ to the state $[x]_i$ of subsystem $i$. The map $\mathbf{\Phi}_{x}$ is \textit{d-localized} if and only if for every subsystem $j$, $[\mathbf{\Phi}_{x}]_{ij}=0\ \forall\ i\not\in\textbf{out}_{j}(d)$. The definition for \textit{d-localized} $\mathbf{\Phi}_u$ is analogous (but maps  disturbance $[w]_{j}$ to the input $[u]_i$).
\end{definition}

When the state system response $\mathbf{\Phi}_{x}$ is $d$-localized, subsystem $i$ only needs to collect the local subset $[\hat{\mathbf{w}}]_{j\in\textbf{in}_i(d)}$ of disturbance estimates $\hat{\mathbf{w}}$
from its $d$-incoming set of neighbors to compute its local disturbance estimate $[\hat{\mathbf{w}}]_{i}$. Similarly, when the input system response $\mathbf{\Phi}_{u}$ is $d$-localized, then only a local subset $[\hat{\mathbf{w}}]_{j\in\textbf{in}_i(d)}$ of the estimated disturbances $[\hat{\mathbf{w}}]$  is needed by subsystem $i$ to compute its control action $[\mathbf u]_i$. Therefore, if $\mathbf{\Phi}_{x}$ and $\mathbf{\Phi}_{u}$ are $d$-localized each subsystem only needs to collect information from its $d$-incoming set to implement the control law and it only needs to share information with its $d$-outgoing set to allow for its neighboring subsystems to implement their local control laws.
\begin{definition}{\label{def: locality_constraints}}
A subspace $\mathcal{L}_d$ enforces a $d$\textit{-locality constraint} if $\mathbf{\Phi}_{x},\mathbf{\Phi}_{u}\in\mathcal{L}_d$ implies that $\mathbf{\Phi}_{x}$ is $d$-localized and $\mathbf{\Phi}_{u}$ is $(d+1)$-localized\footnote{$\mathbf{\Phi}_u$ is enforced to be $(d+1)$-localized because to confine the effects of a disturbance within the region of size $d$, controllers at distance $d+1$ must take action.}.
\end{definition}
\begin{remark}
Despite locality constraints always being an affine constraint, an SLS problem subject to locality constraints might have an empty solution space since not all systems are $d$-localizable. The size of the locality region $d$ is a design parameter, and for the remainder of the paper we assume that system \eqn{LTI} is $d$-localizable for some $d\ll N$. The results presented here extend naturally to virtually localizable systems via the robust variant of SLS \cite{anderson2019system,matni2017scalable}.
\end{remark}

\subsubsection{State and Input Constraints}

Next we present a method for dealing with the robust state and input constraints. We emphasize how the resulting SLS reformulation retains the locality structure of the original problem, which will be key to enabling distribution of the problem across the different subsystems as we will discuss in \sec{robust_dlmpc}.

As stated, problem \eqn{MPC_SLS} is a robust optimization problem since the convex constraint $\mathbf{\Phi}\mathbf{w}\in\mathcal{P}$ has to hold for all $\mathbf{w}\in\mathcal{W}$. Let us define the polytope $\mathcal P$ as $\mathcal P:= \{[\mathbf{x}^\intercal \ \mathbf{u}^\intercal]^\intercal:\ H[\mathbf{x}^\intercal \ \mathbf{u}^\intercal]^\intercal\leq h\}$, and since $\mathbf w$ is a polytopic constraint $\mathcal{W}:=\{\mathbf w:\ G\mathbf w \leq g\}$. A simple duality argument allows for a tractable reformulation of problem \eqn{MPC_SLS} (see \cite{chen2019system} for details and references therein):

\begin{theorem}[Theorem 1 of \cite{chen2019system}]\label{thm:robust_SLS}
Problem \eqn{MPC_SLS} can be reformulated as
\begin{equation} \label{eqn:MPC_SLS_dual}
\begin{aligned}
& \underset{\mathbf{\Phi},\mathbf{\Xi}\geq0}{\emph{minimize}}
& &f(\mathbf{\Phi}\{1\}x_0)\\
& \ \emph{s.t.} &  &\begin{aligned}
& x_0 = x(t),\ Z_{AB}\mathbf{\Phi} = I, \\
& H\mathbf{\Phi}\{1\}x_{0}+\mathbf{\Xi}g\leq h,\ H\mathbf{\Phi}\{2:T\}=\mathbf{\Xi}G,
\end{aligned}
\end{aligned}
\end{equation}
where
\begin{align*}
&G := \emph{\textbf{blkdiag}}(G_{1},...,G_{T-1})~\text{and,} \\
&H := \emph{\textbf{blkdiag}}(H_{x,1},...,H_{x,T},H_{u,1},...,H_{u,T-1}),
\end{align*}
with the constraint $\mathbf{\Xi}\geq0$  satisfied component-wise.
Further, if each of the $\{G_t,H_{x,t},H_{u,t}\}_{t=1}^{T}$ are block-diagonal with structure compatible with subsystem-wise decompositions of $\mathbf{x}$ and $\mathbf{u}$, then imposing that $\mathbf \Phi \in \mathcal L_d$ allows us to take $\mathbf \Xi \in \mathcal L_d$ without loss of generality.
\end{theorem}

Theorem \ref{thm:robust_SLS} allows us to solve the robust MPC problem \eqn{MPC} using standard convex optimization methods, and further preserves the locality structure of the original problem so long as the constraint matrices $H$ and $G$ defining the polytopes $\mathcal{P}$ and $\mathcal{W}$, respectively, have suitable block-diagonal structure.  We formalize the implications of these structural \emph{decoupling} assumptions on the constraint sets, and introduce an analogous decoupling assumption on the objective function, here.

\begin{assumption}\label{assump: locality}
The objective function $f_{t}$ in formulation \eqn{MPC} is such that $f_{t}(x)=\sum f_{t}^i([x]_i,[u]_i)$. The constraint sets in formulation \eqn{MPC} are such that $x\in\mathcal{X}_t=\mathcal{X}_{t}^1\times ... \times \mathcal{X}_{t}^N$, where $x \in \mathcal{X}_t$ if and only if $[x]_i\in\mathcal{X}_{t}^i$ for all $i$ and $t\in\{0,...,T\}$, and idem for $\mathcal{U}_t$.  In particular, this implies that the matrices $H$ and $G$ defining the polytopes $\mathcal{P}$ and $\mathcal{W}$, respectively, are block-diagonal as described in Theorem \ref{thm:robust_SLS}.
\end{assumption}

\begin{remark}
Although Theorem \ref{thm:robust_SLS} focuses on disturbances constrained to polytopic constraint sets, the argument can be extended to any setting for which a closed-form expression for $\sum_{i=0}^N\sup_{[w]_i \in \mathcal{W}^i} e_j^\intercal [H\mathbf{\Phi}]\{i\}[w]_i\leq h$ can be computed for $e_j$ the standard basis unit.  For example, if $\mathcal{W}^i = \{ w \, : \, \| w \| \leq \sigma_w\}$, then this reduces to $\sum_{i=0}^N \sigma_w \|[H\mathbf{\Phi}]\{i\}^\intercal e_j\|_\star\leq h$, for $\|\cdot\|_\star$ the dual norm.  Note however, that for block-diagonal $H$ and localized $\mathbf\Phi$, this constraint will reduce to $\sum_{i : [H\mathbf{\Phi}]\{i\}^\intercal e_j\neq 0} \sigma_w \|[H\mathbf{\Phi}]\{i\}^\intercal e_j\|_\star \leq h$, which can be similarly enforced using only local information.
\end{remark}

\subsection{System Level Synthesis based Robust MPC}
Here we combine the previously introduced locality and state/input robust constraints with the SLS reformulation of the MPC subroutine \eqn{MPC}, and present the robust Distributed and Localized MPC problem. We reformulate subroutine \eqn{MPC} in terms of the SLS variables as in equation \eqn{MPC_SLS}, and impose appropriate $d$-local structural constraints on the \emph{closed-loop system responses} of the system. Hence, both the synthesis and implementation of the control action at each subsystem is localized, i.e., the control policy depends only on state, control, and plant model information from $d$-hop neighbors.

While it was not possible to incorporate locality as expressed in equation \eqn{local_comms} into the original problem \eqn{MPC} in a convex and computationally tractable manner, it is straightforward to do so via the affine constraint $\mathbf{\Phi_x,\mathbf\Phi_u\in\mathcal L_d}$ as per Definition \ref{def: locality}, with the only requirement that some mild compatibility assumptions as per Assumption \ref{assump: locality} between the cost functions, state and input constraints, and $d$-local information exchange constraints are satisfied.

In particular, Assumption \ref{assump: locality} imposes that neither the constraints nor the objective function can couple two subsystems in the network, and therefore they are fully separable among the different subsystems. Notice that even though this assumption is written in terms of the objective function and the constraints of problem  \eqn{MPC}, it directly translates to the objective function and constraint set of the SLS reformulation \eqn{MPC_SLS_dual}. Moreover, this local structure holds for the duality based reformulation of the robust SLS problem described in \thm{robust_SLS}. Therefore, the rDLMPC problem \eqn{MPC} subject to local information constraint \eqn{local_comms} can be written as
\begin{subequations}\label{eqn:robust_DLMPC}
\begin{align}
& \underset{\mathbf{\Phi},\mathbf{\Xi}\geq0}{\text{minimize}}
& &f(\mathbf{\Phi}\{1\}x_0)\\
& \ \text{s.t.} &  & x_0 = x(t),\ Z_{AB}\mathbf{\Phi} = I,\ \mathbf{\Phi},\mathbf{\Xi}\in\mathcal{L}_d,\\
&  &  & H\mathbf{\Phi}\{1\}x_{0}+\mathbf{\Xi}g\leq h,\ H\mathbf{\Phi}\{2:T\}=\mathbf{\Xi}G \label{eqn:robust_DLMPC-const1}.
\end{align}
\end{subequations}

One of the advantages of using the SLS framework is that this approach is automatically optimizing over \emph{closed-loop policies} -- as opposed to open-loop control inputs -- since the optimization is carried out over the closed-loop maps $\mathbf{\Phi}_x$ and $\mathbf{\Phi}_u$. In \cite{amoalonso2020distributed}, we discuss and formalize in detail the advantages of using SLS over other feedback parameterizations such as \cite{goulart2006optimization,furieri2017robust} for distributed MPC. In summary, SLS allows for localized communication constraints \eqn{local_comms} as well as the robust state and input constraints to be transparently enforced through convex constraints on the system responses.  This in turn results in a problem structure that is amenable to using distributed optimization for the computing of the resulting control policies, and as we will show in the next section, sub-controllers only need to collect local information to do so.
\section{Robust Distributed and Localized MPC}\label{sec:robust_dlmpc}

We present a brief overview of the ADMM algorithm, and then show how it can be used to distribute the rDLMPC subroutine \eqn{robust_DLMPC} into local subproblems requiring $d$-local information only. We finish by analyzing the convergence and complexity of the resulting rDLMPC algorithm.

\subsection{The Alternating Method of the Multipliers}
The varying penalty parameter ADMM algorithm~\cite[\S 3.4]{boyd2010distributed} solves problems of the form
\begin{equation*}
\underset{x,y}{\text{minimize}} \ f(x)+g(y)\ \text{s.t.} \ Ax+By=c.
\end{equation*}
The ADMM  iterates are as follows:
\begin{equation*}
\begin{aligned}
& x^{k+1} = \arg\min_x f(x)+\frac{\rho^k}{2}\left\| Ax+By^{k}-c+z^{k}\right\|_{2}^{2}\\
& y^{k+1} = \arg\min_y g(y)+\frac{\rho^k}{2}\left\| Ax^{k+1}+By-c+z^{k}\right\|_{2}^{2}\\
& z^{k+1} = z^{k}+Ax^{k+1}+By^{k+1}-c,
\end{aligned}
\end{equation*}
where
\begin{equation}\label{eqn:rho}
\rho^k:=
\begin{cases}
\tau\rho^k & \text{if } \left\Vert r^{k+1} \right\Vert_2 > \mu_1 \left\Vert s^{k+1} \right\Vert_2, \\
\tau^{-1}\rho^k & \text{if } \left\Vert r^{k+1} \right\Vert_2 < \mu_2 \left\Vert s^{k+1} \right\Vert_2,\\
\rho^{k-1} & \text{otherwise,}
\end{cases}
\end{equation}
with
$r^{k}:= Ax^{k}+By^{k}-c,\ s^{k} := \rho A^\intercal B(y^{k}-y^{k-1}),$
and $\tau,\mu_1,\mu_2>0$. This version of ADMM is known as the varying penalty parameter.

The  algorithm is said to have converged at iteration $k$ if
$\left\Vert r^{k} \right\Vert_2 \leq \epsilon_p$, and $\left\Vert s^{k} \right\Vert_2 \leq \epsilon_d$ for some small primal and dual tolerances, $\epsilon_p>0$ and $\epsilon_d>0$. To guarantee convergence, $\rho$ is set to a constant value $\rho_\text{max}$ after a certain number of iterations. Residual, objective, and dual variable convergence are guaranteed under mild regularity assumptions, assuming a non-empty feasible set (as this implies strong duality holds, as Slater's condition always holds when constraints are polytopes and either the primal or dual problem is feasible).
\begin{assumption}\label{assump: saddle point}
The objective function $f(\mathbf{\Phi}x_0)$ is a closed, proper, and convex function for all choices of $x_0\neq 0$.
\end{assumption}

\subsection{The Robust DLMPC Algorithm}

We will show how the ADMM algorithm can be used to distribute \eqn{robust_DLMPC} into a local subproblem per subsystem. To do so, we exploit the localization of the cost function and the constraints, which by Assumption \ref{assump: locality} imply that
\begin{gather*}
f(\mathbf x,\mathbf u) = \sum_{i = 1}^N f^i([\mathbf x]_i, [\mathbf u]_i), \text{ and} \\
[\mathbf x^\intercal\ \mathbf u^\intercal]^\intercal \in \mathcal P \text{ if and only if } [H]_i\begin{bmatrix}[\mathbf x]_i \\ [\mathbf u]_i \end{bmatrix}\leq [h]_i\ \forall i,
\end{gather*}
with each submatrix $[H]_i$ corresponding to the constraints of subsystem $i$ for $i=1,\dots,N$, i.e., $H=\textbf{blkdiag}([H]_1,\dots,[H]_N)$. Using the definition of the SLS system responses \eqn{Phis}, we can equivalently write these conditions in terms of $\mathbf \Phi$ as
\begin{gather*}
f(\mathbf \Phi) = \sum_{i = 1}^N f^i\big(\mathbf \Phi\{1\}(\mathfrak{r}_i,:)x_0\big), \text{ and} \\
\mathbf \Phi \in \mathcal P \text{ if and only if } [H]_i\mathbf \Phi(\mathfrak{r}_i,:)x_0\leq [h]_i\ \forall i,
\end{gather*}
where $\mathfrak r_i$ is the set of rows in $\mathbf \Phi$ corresponding to subsystem $i$, i.e., those that parametrize $[\mathbf x]_i$ and $[\mathbf u]_i$. These separability features are formalized in the following definition:
\begin{definition}
Given the partition $\{\mathfrak c_1,...,\mathfrak c_k\}$, a functional/set is \emph{column-wise separable} if:
\begin{itemize}
\item For a functional, $g(\mathbf \Phi) = \sum_{i = 1}^k g_i\big(\mathbf \Phi(:,\mathfrak{c}_i)\big)$ for some functionals $g_i$ for $i=1,...,k$.
\item For a set, $\mathbf \Phi \in \mathcal P$ if and only if $\mathbf \Phi(:,\mathfrak{c}_i) \in \mathcal P_i \ \forall i$ for some sets $\mathcal P_i$ for $i=1,...,k$.
\end{itemize}
\end{definition}
An analogous definition exists for \emph{row-wise separable} functionals and sets, where the partition $\{\mathfrak r_1,...,\mathfrak r_k\}$ entails the rows of $\mathbf \Phi$, i.e., $\mathbf \Phi(\mathfrak{r}_i,:)$.

When the objective function and all the constraints of an optimization problem share the same type of separability with respect to a partition of cardinality $k$, then the optimization trivially decomposes into $k$ independent subproblems. However, this is not the case for problem \eqn{robust_DLMPC}, where the constraint \eqn{robust_DLMPC-const1} is neither row-wise nor column-wise separable. To make it amenable for distribution, we propose the following reformulation:
\begin{align}\label{eqn:robust_DLMPC_reformulation}
& \underset{\mathbf{\Phi},\mathbf{\Psi},\mathbf{\Omega}, \mathbf{\Xi}\geq0}{\text{min}}
& &f(\mathbf{\Phi} x_0)\\
& \ \text{s.t.} &
& x_0 = x(t),\ Z_{AB}\mathbf{\Psi} = I,\ \mathbf{\Phi},\mathbf{\Xi},\mathbf{\Psi},\mathbf{\Omega}\in\mathcal{L}_{d}, \nonumber \\
&  &  &\mathbf{\Omega}x_{0}+\mathbf{\Xi}g\leq h, \nonumber \\
&  &   &\begin{bmatrix} I & 0 \\ 0 & H \end{bmatrix} \begin{bmatrix} \mathbf{\Psi}\{1\} & 0 \\ \mathbf{\Psi}\{1\} & \mathbf{\Psi}\{2:T\}\end{bmatrix}= \begin{bmatrix} \mathbf{\Phi} & 0 \\ \mathbf{\Omega} & \mathbf{\Xi} G \end{bmatrix} \nonumber ,
\end{align}
where  $\mathbf{\Phi},\mathbf{\Omega}$ to refer to only the first block-column of the homonymous of the matrices $\mathbf{\Phi}\{1\},\mathbf{\Omega}\{1\}$.

In this reformulation, it is straightforward to distinguish between the column-wise separable components (everything involving $\mathbf\Psi$), and the row-separable components (the objective function and the remaining constraints that do not contain $\mathbf\Psi$). Problems of this nature are known as \emph{partially separable} problems, and are amenable for distribution via ADMM. In particular, we can decompose problem \eqn{robust_DLMPC_reformulation} as
\begin{subequations}\label{eqn:robust_DLMPC_distributed}
\begin{align}
& \begin{aligned}
&[\mathbf{\Phi}^{k+1},\mathbf{\Omega}^{k+1},\mathbf{\Xi}^{k+1}] = \\
&\left\{\begin{aligned}
&\underset{\mathbf{\Phi},\mathbf{\Omega},\mathbf{\Xi}\geq0}{\text{argmin}}
&& \begin{aligned}
& f(\mathbf{\Phi}x_0)+ \frac{\rho}{2}\left\Vert
\tilde{\mathbf{\Phi}}-
\tilde{H}\tilde{\mathbf\Psi}^k + \mathbf{\Lambda}^{k}\right\Vert^{2}_{F}
\end{aligned}
\\
&\;\;\;\text{s.t.}  &&
\begin{aligned}
& x_0 = x(t),\ \mathbf{\Omega}x_{0}+\mathbf{\Xi}g\leq h,\\
&\mathbf{\Phi},\mathbf{\Xi},\mathbf{\Omega}\in\mathcal{L}_{d}
\end{aligned}
\end{aligned}\right\}
\end{aligned} \label{eqn:robust_DLMPC_distributed-row}
\\[5pt]
& \mathbf{\Psi}^{k+1} =
\left\{
\begin{aligned}
&\underset{\mathbf{\Psi}}{\text{argmin}}
&& \begin{aligned}
& \left\Vert
\tilde{\mathbf{\Phi}} -
\tilde{H} \tilde{\mathbf\Psi}
+ \mathbf{\Lambda}^{k}\right\Vert^{2}_{F}
\end{aligned}
\\
&\;\;\;\text{s.t.}
&& Z_{AB}\mathbf{\Psi} = I,\; \mathbf{\Psi}\in\mathcal{L}_{d}
\end{aligned}\right\}
\label{eqn:robust_DLMPC_distributed-column}
\\[5pt]
&\begin{aligned}
\mathbf{\Lambda}^{k+1}   =  \mathbf{\Lambda}^{k}+
\tilde{\mathbf\Phi}^{k+1} - \tilde{H}\tilde{\mathbf\Psi}^{k+1},
\end{aligned}\label{eqn:robust_DLMPC_distributed-lagrange}
\end{align}
\end{subequations}
where for ease of notation we define
$$\tilde{\mathbf{\Phi}} := \begin{bmatrix} \mathbf{\Phi} & 0 \\ \mathbf{\Omega} & \mathbf{\Xi} G \end{bmatrix} \text{ and }
\tilde{H} \tilde{\mathbf\Psi} := \begin{bmatrix} I & 0 \\ 0 & H \end{bmatrix} \begin{bmatrix} \mathbf{\Psi}\{1\} & 0 \\ \mathbf{\Psi}\{1\} & \mathbf{\Psi}\{2:T\}\end{bmatrix} $$

In equations \eqn{robust_DLMPC_distributed}, iterate \eqn{robust_DLMPC_distributed-row} is row-wise separable,\eqn{robust_DLMPC_distributed-column} is column-wise separable, and \eqn{robust_DLMPC_distributed-lagrange} is both row and column-wise separable. We take advantage of the structure of these subproblems and separate them with respect to a row and column partition induced by the subsystem-wise partitions of the state and control inputs, $[\mathbf x]_i$ and $[\mathbf u]_i$ for each subsystem $i$. Each of these row and column subproblems resulting from the distribution across subsystems can be solved independently and in parallel, where each subsystem solves for its corresponding row and column partition. Furthermore, since locality constraints are imposed, the decision variables $\mathbf\Phi,\mathbf\Psi,\mathbf\Xi,\mathbf\Lambda$ have a sparse structure. This means that the length of the rows and columns that a subsystem solves for is much smaller than the length of the rows and columns of $\mathbf{\Phi}$. For instance, when considering the column-wise subproblem evaluated at subsystem $i$, the $j^\text{th}$ row of the $i^\text{th}$ column partitions of $\mathbf\Phi_x$ and $\mathbf\Phi_u$ is nonzero only if $j\in\cup_{k\in\textbf{out}_i(d)}\mathfrak{r}_k$ and $j\in\cup_{k\in\textbf{out}_i(d+1)}\mathfrak{r}_k$ respectively. An algorithm to find the relevant components for each subsystem rows and columns can be found in Appendix A of \cite{wang2018separable}. Therefore, the subproblems that subsystem $i$ solves for are:
\begin{subequations}\label{eqn:robust_DLMPC_local}
\begin{align}
& \begin{aligned}
&[[\mathbf{\Phi}]_{i_r}^{k+1},[\mathbf{\Omega}]_{i_r}^{k+1},[\mathbf{\Xi}_{i_r}^{k+1}]] = \\
&\left\{\begin{aligned}
&\underset{[\mathbf{\Phi}]_{i_r},[\mathbf{\Omega}]_{i_r},[\mathbf{\Xi}]_{i_r}\geq0}{\text{argmin}}
&& \begin{aligned}
& f([\mathbf\Phi]_{i_r}[x_0]_{i_r})+ \\
&\frac{\rho}{2} \left\Vert {[\tilde{\mathbf{\Phi}}]_{i_r}}-[\tilde{H}]_{i_r}[\tilde{\mathbf\Psi}]_{i_r}^k + [\mathbf{\Lambda}]_{i_r}^{k}\right\Vert^{2}_{F}
\end{aligned}
\\
&\;\;\;\text{s.t.}  &&
\begin{aligned}
& [\mathbf{\Omega}]_{i_r}[x_{0}]_{i_r}+[\mathbf{\Xi}]_{i_r}[g]_{i_r}\leq [h]_{i_r}
\end{aligned}
\end{aligned}\right\}
\end{aligned} \label{eqn:robust_DLMPC_local-row}
\\[5pt]
& [\mathbf{\Psi}]_{i_c}^{k+1} =
\left\{
\begin{aligned}
&\underset{[\mathbf{\Psi}]_{i_c}}{\text{argmin}}
&& \begin{aligned}
& \left\Vert
{[\tilde{\mathbf{\Phi}}]_{i_c}} -
[\tilde{H}]_{i_c} [\tilde{\mathbf\Psi}]_{i_c}
+ [\mathbf{\Lambda}]_{i_c}^{k}\right\Vert^{2}_{F}
\end{aligned}
\\
&\;\;\;\text{s.t.}
&& [Z_{AB}]_{i_c}[\mathbf{\Psi}]_{i_c} = [I]_{i_c}
\end{aligned}\right\}
\label{eqn:robust_DLMPC_local-column}
\\[5pt]
& [\mathbf{\Lambda}]_{i_c}^{k}  =  [\mathbf{\Lambda}]_{i_c}^{k}+
[\tilde{\mathbf{\Phi}}]_{i_c}^{k+1} - [\tilde{H}]_{i_c}[\tilde{\mathbf\Psi}]_{i_c}^{k+1}, \label{eqn:robust_DLMPC_local-lagrange}
\end{align}
\end{subequations}
where to simplify notation we denote as $[\mathbf\Phi]_{i_r}$ the submatrix of $\mathbf\Phi$ formed by the nonzero components of the relevant rows for subsystem $i$, $\mathfrak{r}_i$, and as $[\mathbf\Phi]_{i_c}$ the submatrix of $\mathbf\Phi$ formed by the nonzero components of the relevant columns for subsystem $i$, $\mathfrak{c}_i$. We use a similar bracket notation for the vectors and matrices that multiply the decision variables to indicate that those are just composed from the relevant components of their global versions.

The following lemma will be useful later for speeding up the ADMM iterates.
\begin{lemma}\label{lem:closed-form}
Let $z^{\star}(M,v,P,q) := \argmin\limits_{z} \left\Vert Mz-v \right\Vert_2^2$ s.t. $Pz=q$. Then
$$\begin{bmatrix} z^{\star} \\ \mu^{\star} \end{bmatrix} = \begin{bmatrix}MM^\intercal & P^\intercal \\ P & 0\end{bmatrix}^\dagger \begin{bmatrix} M^\intercal v \\ q \end{bmatrix},$$
where  $\dagger$ denotes pseudo-inverse, is the optimal solution and $\mu^{\star}$ is the corresponding optimal Lagrange multiplier.
\end{lemma}
\begin{proof}
The proof follows directly from applying the KKT conditions to the optimization problem. By the stationarity condition, $M^\intercal M z^{\star}-M^\intercal v + P^\intercal \mu^{\star} = 0$, where $z^{\star}$ is the solution to the optimization problem and $\mu^{\star}$ the optimal Lagrange multiplier vector. From the primal feasibility condition, $Pz^{\star} = q$. Hence, $z^{\star}$ and $\mu^{\star}$ are computed as the solution to this system of two equations.
\end{proof}

Applying lemma~\ref{lem:closed-form}, the solution to  \eqn{robust_DLMPC_local-column} is:
\begin{subequations}\label{eqn:closed_form}
\begin{align}
\hspace{-2mm}& [\mathbf{\Psi}\{1\}]_{i_c}^{k+1} = z^{*} \Bigg(
\begin{bmatrix} [I]_{i_c}\ \\ [H]_{i_c}\ \end{bmatrix}
\begin{bmatrix}
[\mathbf{\Phi}]_{i_c}+[\mathbf{\Lambda_1}]_{i_c}\\ [\mathbf{\Omega}]_{i_c}+[\mathbf{\Lambda_2}]_{i_c}
\end{bmatrix},
[Z_{AB}]_{i_c},
\begin{bmatrix} [I]_{i_c}\ \\ 0 \end{bmatrix}
\Bigg)\label{Psi_1}
\\
\hspace{-2mm}& [\mathbf{\Psi}\{2:T\}]_{i_c}^{k+1} =  z^{*} \Bigg(
[H]_{i_c}
[\mathbf{\Xi}G]_{i_c}+[\mathbf{\Lambda}_3]_{i_c},
[Z_{AB}]_{i_c},
\begin{bmatrix} 0 \\ [I]_{i_c} \end{bmatrix}
\Bigg) \label{Psi_rest}
\end{align}
\end{subequations}
where we split $\mathbf\Lambda$ into block matrices consistent with the structure of $\tilde{\mathbf\Phi}$, i.e. $\mathbf\Lambda = \begin{bmatrix}\mathbf{\Lambda_1} & 0 \\ \mathbf{\Lambda_2} & \mathbf{\Lambda_3}\end{bmatrix}.$

Notice that the number of nonzero components for both the rows and the columns is much smaller than the size of the network $N$ since it is determined by the size of the local neighborhood $d$ through the locality constraints. In turn, this implies that the subsystem only requires small submatrices from the plant matrices and the constraints $A,B,H,$ etc, to perform the computations. Therefore, the robust MPC subroutine \eqn{robust_DLMPC} can be solved in parallel in a distributed manner across the subsystems of the network, where each solves for a local patch of the system responses using local information only. In \alg{rDLMPC} we present the rDLMPC algorithm that each sub-controller executes.
\vspace{-2mm}
\setlength{\textfloatsep}{0pt}
\begin{algorithm}[ht]
\caption{Subsystem $i$ rDLMPC implementation}\label{alg:rDLMPC}
\begin{algorithmic}[1]
\Statex \textbf{input:} $\epsilon_p, \epsilon_d, \mu_1, \mu_2, \tau, \rho_\text{max}>0$.
\State Measure local state $[x(t)]_{i}$, $k\leftarrow0$.
\State Share the measurement with neighbors in $\textbf{out}_{i}(d)$.
\State Solve optimization problem \eqn{robust_DLMPC_local-row}.
\State Share $[\tilde{\mathbf{\Phi}}]_{i_{r}}^{k+1}$ with $\textbf{out}_{i}(d)$. Receive the corresponding $[\tilde{\mathbf{\Phi}}]_{j_{r}}^{k+1}$ from $\textbf{in}_{i}(d)$ and build $[\tilde{\mathbf{\Phi}}]_{i_{c}}^{k+1}$.
\State Solve optimization problem \eqn{robust_DLMPC_local-column} via the closed form solution \eqn{closed_form}.
\State Share $[\tilde{\mathbf{\Psi}}]_{i_{c}}^{k+1}$ with $\textbf{out}_{i}(d)$. Receive the corresponding $[\tilde{\mathbf{\Psi}}]_{j_{c}}^{k+1}$ from $\textbf{in}_{i}(d)$ and build $[\tilde{\mathbf{\Psi}}]_{i_{r}}^{k+1}$.
\State Perform the multiplier update step \eqn{robust_DLMPC_local-lagrange}.
\State \textbf{if} {$\left\Vert[\tilde{\mathbf{\Phi}}]_{i_{r}}^{k+1}-[\tilde{H}]_{i_{r}}[\tilde{\mathbf{\Psi}}]_{i_{r}}^{k+1}\right\Vert_F\leq\epsilon_{p}$
\Statex  and $\left\Vert[{\mathbf{\Psi}}]_{i_{r}}^{k+1}-[{\mathbf{\Psi}}]_{i_{r}}^{k}\right\Vert_F\leq\epsilon_{d}$}\textbf{:}
\Statex $\;\;$ Apply control action $[u_0]_i = [\Phi_{u,0}[0]]_{i_{r}}[x_0]_{\mathfrak{r_i}}$, and return to step 1.
\Statex \textbf{else:}
\Statex $\;\;$ Set $k\leftarrow k+1$, update $\rho$ according to \eqn{rho}, and return to step 3.
\end{algorithmic}
\end{algorithm}
\vspace{-4mm}
\subsection{Convergence and Complexity Analysis}

\subsubsection{Complexity Analysis}
We focus on the computational overhead since communication (as described by steps 2, 4 and 6, of \alg{rDLMPC}) is limited to information exchange between $d$-local neighbors. Steps 3, 5 and 7 dominate the computational complexity of the algorithm. The optimization problems in step 5 has a closed form solution (c.f. \lem{closed-form}), and step 7 requires only simple matrix operations. As a result, the computational overhead is significantly reduced. In particular, their evaluation reduces to addition and multiplication of matrices of dimension $O(d^2T^2)$, independent of the global size of the network $N$. Step 3 has to be computed by means of an optimization solver, since no closed-form solution exists and finding an explicit solution is no longer efficient (in contrast to  the nominal case  \cite{amoalonso2020explicit}) given the structure of the robust constraints. However, the size of the problem remains much smaller than the size of the network, specially when $d\ll N$. Each local row-wise subproblem optimizes over $O(d^2T^2)$ optimization variables subject to $O(dT)$ constraints. Note that in the nominal case \cite{amoalonso2020distributed}, complexity of the DLMPC algorithm scales with order $O(d^2T)$. The additional $T$ factor is due to the fact that in the noise-free setting we only work with the first block-column of the system responses, whereas in the robust case all the matrices composing the system responses become relevant. As in the nominal case, both the local radius $d$ and the time horizon $T$ are independent design parameters.

\subsubsection{Convergence Analysis}
Convergence of \alg{rDLMPC} follows directly from convergence results from ADMM \cite{boyd2010distributed} which require
Assumption \ref{assump: saddle point} and strong duality, which follows directly since all constraint sets are polytopes.

\subsubsection{Recursive Feasibility and Stability}
As is standard, recursive feasibility and stability can be guaranteed by setting a robust invariant set as the terminal set and a global Lyapunov function as the terminal cost \cite{borrelli2016model, lofberg2003minimax}.  We leave finding methods for efficiently computing decoupled robust invariant sets and terminal costs satisfying Assumption \ref{assump: locality} by exploiting locality to future work.

\section{Simulation Experiments}\label{sec:simulations}

We evaluate performance and computational overhead of \alg{rDLMPC}. We choose a chain system composed of $N$ subsystems where each has one scalar state. The local dynamics are given by:
\begin{equation}\label{eqn:simulation}
[x(t + 1)]_i = \alpha \big([x(t)]_i + \sum_{j=i\pm1}\kappa [x(t)]_j \big) + \beta_i[u(t)]_i + [w(t)]_i,
\end{equation}
where $\alpha=0.8$, $\kappa=2$, $\beta_i=1$ for $i=\{1,3,5,6,8,10\}$ and $\beta_i = 0$ for $i=\{2,4,7,9\}$ (notice the symmetry in the chain). We set random uniform noise to $\left\Vert [w]_i \right\Vert \leq 1$ for all $i$. The system is subject to constraints $[x^{min}(t)]_i \leq [x(t)]_i \leq [x^{max}(t)]_i $ for all $t$. Simulations are performed with SLS toolbox \cite{slstoolbox} and code to reproduce the experiments is at \url{https://github.com/unstable-zeros/dl-mpc-sls}. Standard values for the ADMM parameters are $\tau=1.5$, $\mu_1=\mu_2=10$, $\rho_{\text{max}}=5$. Some variation exists across different scenarios, values can be found in the code.

\subsection{Performance}

We analyze the trajectory obtained when computing the control input with \alg{rDLMPC} versus when it is computed with the nominal DLMPC Algorithm in \cite{amoalonso2020distributed}. We choose $N=10$, $d=3$, $T=5$, and we simulate the system over a time horizon of $20$ time steps. We set $[x^{max}(t)]_i = 1.5$ and $[x^{min}(t)]_i = -1.5$ for all nodes with actuation, i.e., $\beta_i=1$, and  $[x^{max}(t)]_i = 20$ and $[x^{min}(t)]_5  = -20$ for all other nodes. We run simulations for $5$ different random noise realizations, and plot a sample resulting trajectory for subsystem $3$ in \fig{dynamics}. In the presence of disturbance, the state of the subsystem remains within the bounds for all times when the control input is computed via \alg{rDLMPC}, whereas it significantly violates the bounds when the control input is computed via the nominal scheme of DLMPC \cite{amoalonso2020distributed}. Moreover, when no disturbance is present both algorithms provide the same result, which illustrate that \alg{rDLMPC} is a generalization of \cite{amoalonso2020distributed} in the case of affine constraints. The algorithm converges to the same value as CVX in all cases.

\begin{figure}
\centering
\includegraphics[scale=.54]{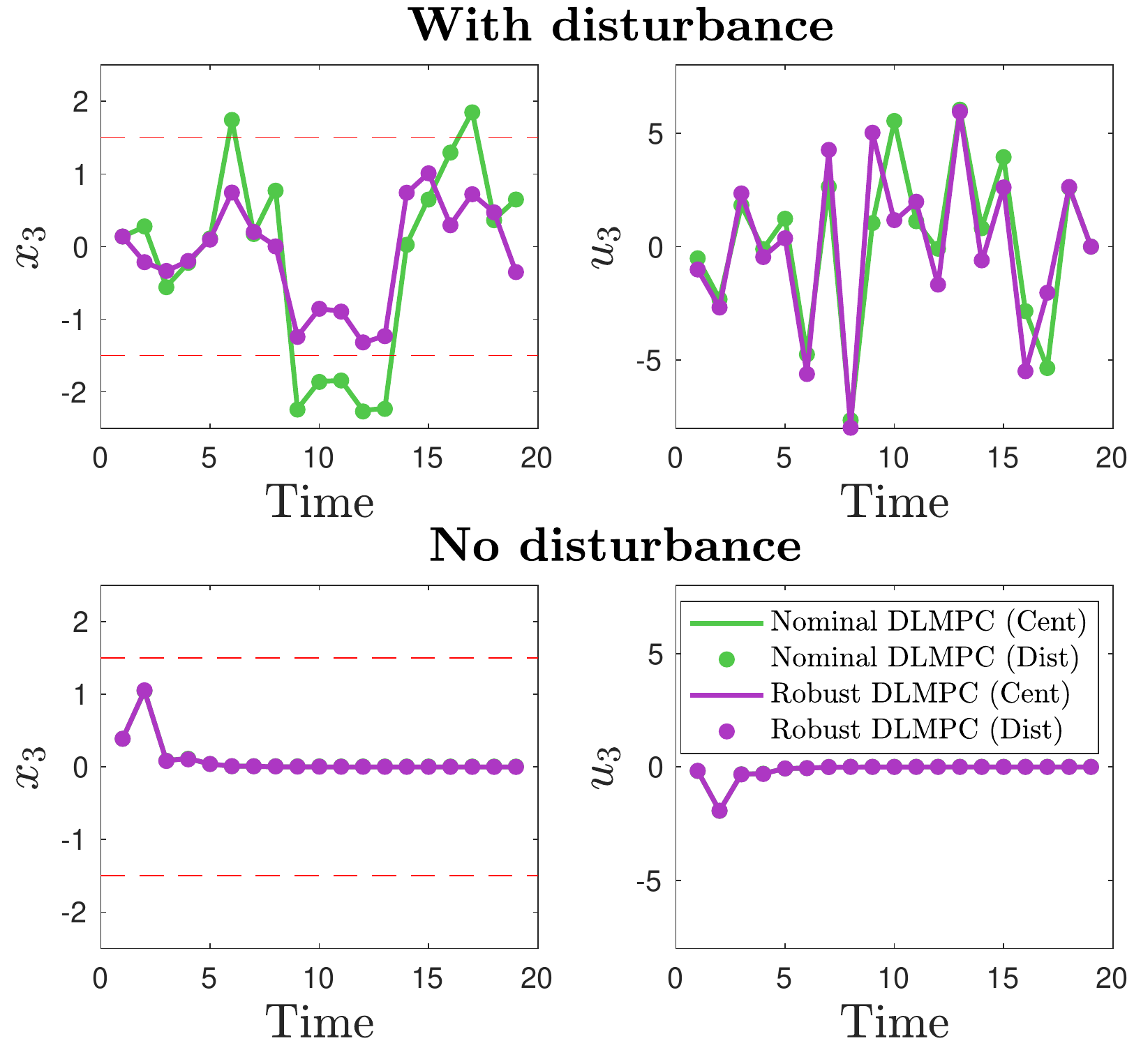}
\caption{Comparison of state and input $3$ when controlled with nominal versus robust DLMPC, in the presence of noise (top) and without noise (bottom).}
\label{fig:dynamics}
\end{figure}

We also investigate the conservativeness of \alg{rDLMPC} through a performance comparison with norminal DLMPC \cite{amoalonso2020distributed}. We present in \tab{cvl} the different values of the average cost (and standard deviation) after running for the $5$ different disturbance realizations. We observe that the cost relative difference between the two algorithms is around $30\%$. Generally, rDLMPC requires a more aggressive control action at the beginning, which leads to higher costs. But since in cases where the disturbance leads to constraint violations, the state in DLMPC reaches higher values and results in additional control actions to correct the deviations. The resulting performance of the rDLMPC algorithm is comparable to its nominal counterpart.

\begin{table}
\centering
\caption{Average cost and standard deviation of DLMPC in \eqn{simulation}. }
\label{tab:cvl}
\renewcommand{\arraystretch}{1.25}
\begin{tabular}{|
@{}>{\centering}m{0.245\columnwidth}@{}|
@{}>{\centering}m{0.245\columnwidth}@{}|
@{}>{\centering}m{0.245\columnwidth}@{}|
@{}>{\centering}m{0.245\columnwidth}@{}|}
\hline
\cline{1-4}
Nominal MPC &  Nominal MPC & Robust MPC & Robust MPC \tabularnewline
(Alg. I in \cite{amoalonso2020distributed}) & (via CVX) & (Alg. I ) & (\eqn{robust_DLMPC} via CVX) \tabularnewline
\hline
\hline
$1344^{*}\pm160$ & $1343^{*}\pm160$ & $1804\pm314$ & $1807\pm323$ \tabularnewline \hline

\end{tabular}
\newline
\\ $*$ denotes constraint violation
\end{table}

\subsection{Computational overhead}

We analyze  \alg{rDLMPC}, and  compare it with the overhead of its nominal counterpart Algorithm I in \cite{amoalonso2020distributed}, both when run with the explicit solution \cite{amoalonso2020explicit} as well as when an optimization solver is needed. We perform simulations with varying parameters, and we analyze how runtime is affected. In particular, we choose a range of values for the size of the network $N$ as well as for the size of the locality region $d$.
Simulation results are presented in \fig{runtime}. As predicted by the complexity analysis in \sec{robust_dlmpc}, runtime is not dominated by the size of the network and it does increase with the size of the locality region. The slight increase observed with the size of the network stabilizes for large network sizes, and this is a feature of ADMM previously observed in \cite{conte2012computational}. The increase with the size of the network is more apparent than in the nominal case because in that case complexity was $O(d^2T)$ while in the robust case we have $O(d^2T^2)$ due to the need for additional optimization variables. Therefore any increase in $d$ will be magnified by an additional factor of $T$.

When comparing runtime values for the different schemes, the nominal approach is more efficient since runtime is faster by approximately two orders of magnitude. However, part of this efficiency is due to the use of an explicit solution, which as demonstrated in \cite{amoalonso2020explicit} provides a significant reduction in runtime. When using the Gurobi solver \cite{gurobi} for the nominal approach, runtimes are faster than the ones of rDLMPC by one order of magnitude. This indicates that the fact that the rDLMPC has significantly more variables than its nominal counterpart accounts for a 10-fold runtime increase, where the other 10-fold is just due to the fact that no explicit solution is being used in the rDLMPC computation.

\begin{figure}
\centering
\includegraphics[scale=0.48]{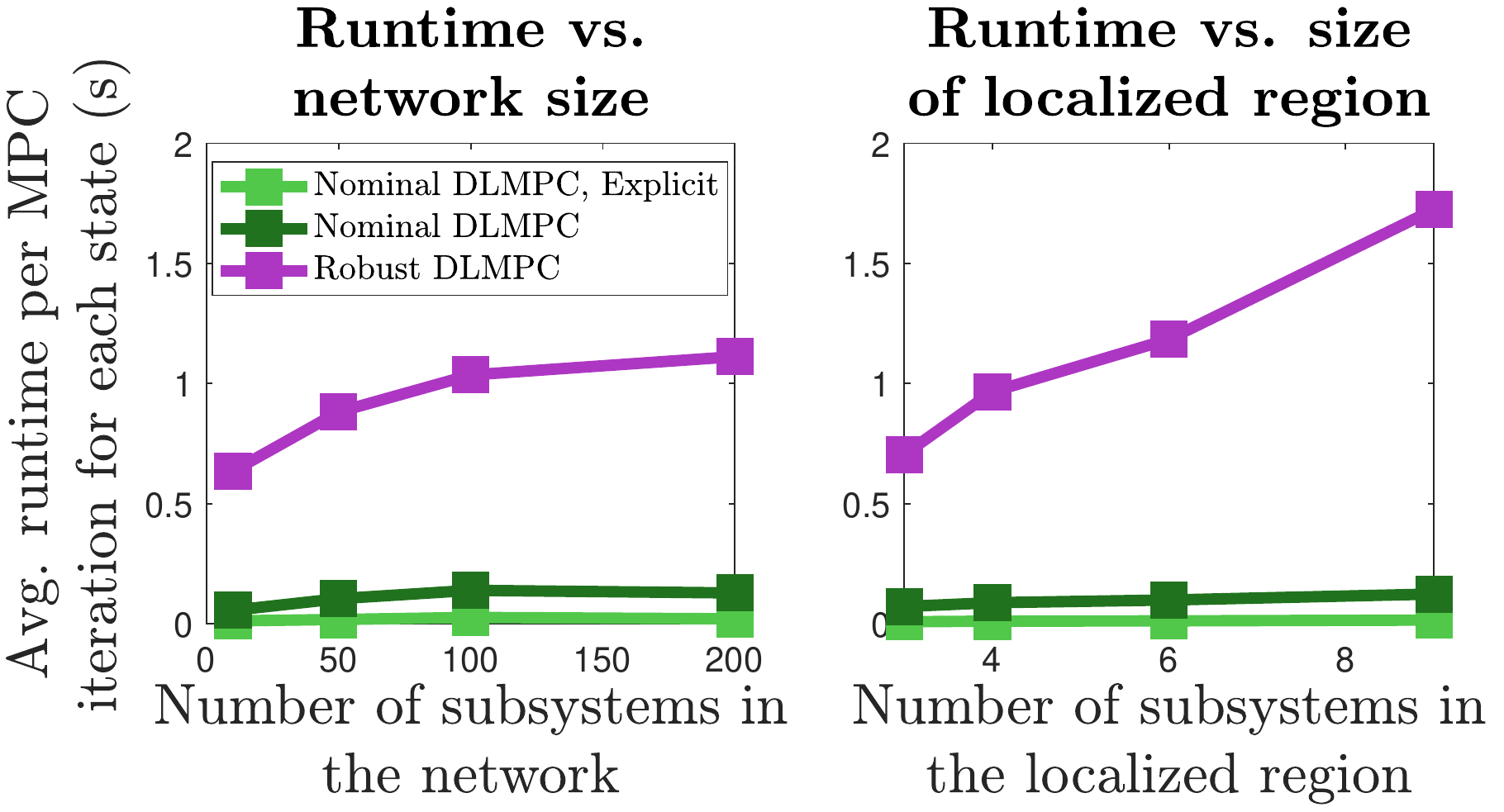}
\caption{On the left, runtime of each case for different network sizes $N$. On the right, the runtime of each case for different localized region sizes $d$.}
\label{fig:runtime}
\end{figure}

\section{Conclusion and Future Directions}\label{sec:conclusion}

We provide a robust closed-loop MPC approach that is distributed and localized in both  synthesis and implementation.
Complexity of the algorithm is dominated by the size of the local region, and we corroborate through simulation that this approach is suitable for large-scale networks. To be best of our knowledge, this is the first DMPC algorithm that allows for the distributed synthesis of robust closed-loop policies. Future work will investigate exploiting locality for the scalable distributed computation of robust invariant sets and terminal cost functions.

\bibliographystyle{IEEEtran}
\bibliography{Test}

\end{document}